\documentclass[12pt]{article}
\usepackage{latexsym,amsmath,amsthm,verbatim,ifthen,amssymb}
\usepackage{mathstuff}
\usepackage[all]{xy}
\usepackage{caption}

\usepackage{transparent}

\usepackage[textwidth=3.5cm]{todonotes}

\newtheorem{theorem}{Theorem}

\newtheorem{lemma}[theorem]{Lemma}

\theoremstyle{definition}
\newtheorem*{definition*}{Definition}
\newtheorem*{example*}{Example}
\newtheorem*{remark*}{Remark}

\makeatletter
\newcommand{\subjclass}[2][2010]{%
  \let\@oldtitle\@title%
  \gdef\@title{\@oldtitle\footnotetext{#1 \emph{Mathematics Subject Classification:} #2.}}%
}
\newcommand{\keywords}[1]{%
  \let\@@oldtitle\@title%
  \gdef\@title{\@@oldtitle\footnotetext{\emph{Key words and phrases:} #1.}}%
}
\makeatother

\begin{document}
\title{Topological complexity and efficiency of motion planning algorithms}
\author{Z. B\l{}aszczyk\footnote{Supported by the National Science Centre grant 2014/12/S/ST1/00368.}, J. Carrasquel\footnote{Supported by the Belgian Interuniversity Attraction Pole (IAP) within the framework ``Dynamics, Geometry and Statistical Physics'' (DYGEST P7/18).}} 
\date{}

%%% ----------------------------------------------------------------------
\maketitle
%%% ----------------------------------------------------------------------
\begin{abstract}
We introduce a variant of Farber's topological complexity, defined for smooth compact orientable Riemannian manifolds, which takes into account only motion planners with the lowest possible ``average length'' of the output paths. We prove that it never differs from topological complexity by more than $1$, thus showing that the latter invariant addresses the problem of the existence of motion planners which are ``efficient''.
\end{abstract}

%%% ----------------------------------------------------------------------
\maketitle
%%% ----------------------------------------------------------------------

\section{Introduction}
A \emph{motion planner} in a topological space $X$ is a section of the fibration $\pi\colon X^I\rightarrow X\times X$ given by $\pi(\gamma):=\big(\gamma(0),\gamma(1)\big)$. 
% We emphasise that we do not assume \textit{a priori} that motion planners are continuous. 
If $X$ is the configuration space of a mechanical system $S$ (i.e. the space of all of its possible states), the space $X^I$ of continuous paths in $X$ can be interpreted as the \emph{space of motions} of~$S$, and a section of $\pi$ is then an algorithm describing how to navigate between any two given states of $S$.\\

The study of motion planners in the above setting was initiated by Farber \cite{Farber03, Farber04, Farber08}. He observed that a continuous motion planner on $X$ exists if and only if $X$ is contractible. This resulted in the introduction of the following invariant, which gives a way of measuring complexity of the motion planning problem.

\begin{definition*} A family $\sigma=\{\sigma_i\colon G_i\rightarrow X^I\}_{i=0}^m$ of continuous local sections of~$\pi$ is called an \emph{$m$-motion planner} on $X$ if:
\begin{enumerate}
\item[(1)] each \textit{domain of continuity} $G_i$ is a locally compact subset of $X \times X$, 
\item[(2)] $G_i\cap G_j = \emptyset$, $i\neq j$, and
\item[(3)] $X\times X = G_0 \cup G_1 \cup \cdots \cup G_m$.
\end{enumerate} 
\textit{Topological complexity} of $X$, denoted $\tc(X)$, is the minimal integer $m\geq 0$ such that there exists an $m$-motion planner on $X$. 
\end{definition*}

In the remaining part of the paper, we take the term ``motion planner'' to mean an $m$-motion planner for some $m \geq 0$. We refer the reader to \mbox{\cite[Chapter 4]{Farber08}} for an elaboration of the notion of topological complexity. (In particular, we note that $\tc$ is typically defined differently. However, if $X$ is an Euclidean neighbourhood retract, which is the only case we will be interested in, the definitions coincide.)

\begin{figure}[!h]
\captionsetup{width=0.85\textwidth}
\begin{center}
%\depth\svgwidth{.5\textwidth}
\scalebox{0.85}{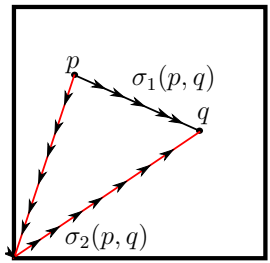}
\caption{\small Paths between states $p$ and $q$ issued by two different $0$-motion planners, $\sigma_1$ and $\sigma_2$. The first one is clearly the more efficient one and, intuitively, the most efficient it can be. The question is, how to make this distinction in more complicated situations?}
\end{center}
\end{figure}

A shortcoming of Farber's approach to complexity of the motion planning problem is that it does not take into consideration any notion of efficiency, e.g. measured in terms of covered distance or spent energy. It is very natural that, given a motion planner, one would like to somehow quantify its efficiency and then, possibly even more urgently, understand how far-off of the most efficient planner it is.\\

The aim of this note is to show that $\tc$ actually addresses the problem hinted at above. In order to do this, we introduce the notion of \emph{efficient topological complexity}, denoted $\eTC$, which takes into account only motion planners with the lowest possible ``average length'' of paths, and then prove that it never differs from $\tc$ by too much, at least for nice spaces:

\begin{theorem}\label{thm:main}
If $X$ is a smooth closed orientable Riemannian manifold, then 
\[\tc(X) \leq \eTC(X) \leq \tc(X)+1.\]
\end{theorem}

\section{Efficient topological complexity}

Fix once and for all a smooth compact orientable Riemannian manifold $X$ and write $d$ for its Riemannian metric. Given a path $\alpha \in X^I$, let $\ell(\alpha)$ denote its length, understood in the metric sense for paths which are merely continuous.
% i.e.
% \[ \ell(\alpha) := \int_0^1 |\alpha'(t)|dt.\]
We do not assume that $\alpha$ is rectifiable, hence it is possible that $\ell(\alpha)=\infty$.

\begin{definition*}
\begin{enumerate}
\item[(1)] The \emph{length} of a motion planner $\sigma\colon X\times X\rightarrow X^I$ is 
\[\ell(\sigma):=\int_{X\times X}\ell\circ\sigma.\]
Note that each domain of continuity of $\sigma$ is measurable and thus $\ell(\sigma)$ is well-defined. Moreover, it is clear that $\int_{X\times X} d \leq \ell(s)$.
\item[(2)] The \textit{efficient topological complexity} of $X$, denoted $\eTC(X)$, is the minimal integer $m\geq 0$ such that there exists an $m$-motion planner $\sigma$ on $X$ with $\ell(\sigma)=\int_{X\times X} d$. Such a motion planner $\sigma$ will be called \textit{efficient}.
\end{enumerate}
\end{definition*}

\noindent It is not \textit{a priori} clear  whether efficient motion planners always exist. This will follow from our proof of Theorem \ref{thm:main}, which we briefly prepare for now.\\

Additionally assume that $X$ has no boundary. Write $\mathcal{U}_p$ for the maximal normal neighbourhood in $T_pX$ and $\cut(p)$ for the cut locus of a point $p \in X$. Then $\exp_p(\mathcal{U}_p) = X \setminus \cut(p)$ and $\exp_p \colon \mathcal{U}_p \to X \setminus \cut(p)$ is a diffeomorphism \cite[Proposition 28.2]{Postnikov}. Let 
\[ \mathcal{V} := \big\{(p,q) \in X \times X \,|\, q \notin \cut(p)\big\}.\]

\begin{lemma}\label{lemma:exp}
\begin{enumerate}
\item[\textnormal{(1)}] The map $\exp \colon \bigcup_{p\in M} \{p\}\times \mathcal{U}_p \to \mathcal{V}$ given by 
\[ \exp(p,v) := \big(p, \exp_p(v)\big)\] 
is a diffeomorphism. In particular, $\mathcal{V} \subseteq X\times X$ is an open subset.%\znote{Why??}
\item[\textnormal{(2)}] The complement of $\mathcal{V}$ in $X\times X$ is a measure-zero subset.
\end{enumerate}
\end{lemma}

\begin{proof}
(1) Since $\exp_p \colon \mathcal{U}_p \to X \setminus \cut(p)$ is a diffeomorphism for any $p \in X$, $\exp$ is a bijection and, furthermore, its derivative is invertible at any point $(p,v) \in \bigcup_{p\in M} \{p\}\times \mathcal{U}_p$. Consequently $\exp$ is a a bijective local diffeomorphism, hence a diffeomorphism. 

(2) Since $(X\times X) \setminus \mathcal{V} = \big\{(p,q)\in X\times X \;|\; q \in \cut(p) \big\}$ and $\cut(p)$ is a measure-zero subset for any $p \in X$ \cite[Lemma 3.96]{GallotHulinLafontaine}, the conclusion follows immediately from \cite[Section 42, Theorem 1]{Berberian}.
\end{proof}

We can now give the proof of our main result.\\

\noindent\textit{Proof of Theorem \ref{thm:main}.} Clearly, $\tc(X) \leq \eTC(X)$. We will show that $\eTC(X) \leq \tc(X)+1$. Let $\tc(X)=m-1$ and choose an $(m-1)$-motion planner $\{\omega_i\colon G_i\rightarrow X^I\}_{i=1}^m$ on $X$. Set 
$$P_{0}:= \mathcal{V} = \big\{(p,q) \in X\times X \,|\, q \notin \cut(p)\big\}$$
and define $\sigma_0\colon P_0 \to X^I$ by assigning
\[ \sigma_0(p,q)(t) := \exp\!\big(p, t\cdot\textnormal{proj}_2\big(\!\exp^{-1}(p,q)\big)\big),\]
where $\rm{proj}_2$ is the projection onto the second coordinate. Note that $\sigma_0(p,q)$ is the unique minimal geodesic from $p$ to~$q$, so that $\ell\big(\sigma_0(p,q)\big)=d(p,q)$. It follows from Lemma \ref{lemma:exp} that $P_0 \subseteq X \times X$ is locally compact and $\sigma_0 \colon P_0 \to X^I$ is continuous. Now set, for $i=1,\ldots,m$,
\begin{itemize}
\item $P_i := (X\times X \setminus P_0) \cap G_i$, and 
\item $\sigma_i := {\omega_i}_{|P_i}$.
\end{itemize}
Then $\sigma = \{\sigma_i \colon P_i \to X^I\}_{i=0}^{m}$ constitutes an $m$-motion planner on~$X$. Again by Lemma \ref{lemma:exp}, the complement of $P_0$ is a measure-zero subset, hence so are the sets $P_i$, $i=1,\ldots, m$. Therefore 
$$\ell(\sigma) = \int_{X\times X} \ell\circ\sigma= \int_{P_0} \ell\circ\sigma_0=\int_{P_0} d=\int_{X\times X}d,$$
which concludes the proof.$\hfill\square$\\

\begin{remark*}
The proof of Theorem \ref{thm:main} shows that in order to estimate topological complexity of $X$, it is enough to understand how to motion plan between points $p$, $q \in X$ with $q\in \cut(p)$. This observation can be formalized through the notion of \textit{relative topological complexity}. Namely, if $A\subseteq X\times X$, then $\tc_X(A)$ is expressed in terms of local  sections of the fibration $\pi^{-1}(A) \to A$. Therefore, by \cite[Proposition 4.24]{Farber08}, setting $\mathcal{V}^c=X\times X\setminus \mathcal{V}$, we obtain
\[\tc_X(\mathcal{V}^c)\le\tc(X)\le \eTC(X)\le \tc_X(\mathcal{V}^c)+1.\] 
This is, in fact, Farber's \cite[Example 4.8]{Farber08} approach to motion planners on spheres: recall that if $S^n$ is embedded in $\mathbb{R}^{n+1}$ in the usual manner, the cut locus of any point $p \in S^n$ consists precisely of the antipode of~$p$. The difficulty thus boils down to estimating $\tc_{S^n}\big(\big\{(p,-p) \,|\, p \in S^n\big\}\big)$. 
\end{remark*}

Theorem \ref{thm:main} shows that, perhaps a little surprisingly, $\eTC(X)$ depends on the choice of a Riemannian metric on $X$ only in a very restricted manner. A natural question to consider is whether it depends on that choice at all? The following simple example sheds some light on this problem in the case when~$X$ has a non-empty boundary. (Which, admittedly, is not covered by Theorem \ref{thm:main}.) 

\begin{example*}
Let $D^2$ be the two-dimensional unit disk in $\mR^2$. Straight line segments give rise to a continuous efficient motion planner on $D^2$, hence $\eTC(D^2)=\tc(D^2)=0$. Now consider $D^2$ embedded in $\mR^3$ as the upper hemisphere $S^2_+$ of the two-dimensional unit sphere. Suppose that $\eTC(S^2_+)=0$. Then there exists a continuous motion planner $\sigma \colon S^2_+ \times S^2_+ \to (S^2_+)^I$ with $\int_{X\times X}(\ell\circ\sigma-d)=0$. This and continuity of $\ell\circ\sigma-d$ imply that $\ell\circ\sigma=d$. Thus $\sigma(p,q)$ traverses the arc of a minimal geodesic from $p$~to $q$ for all $p$, $q \in S^2_+$ by \mbox{\cite[Chapter 3, Corollary 3.9]{DoCarmo}}. 
This, however, is absurd, because such a motion planner cannot be continuous on the set $A$ of pairs of antipodal points from the boundary circle. On the other hand, it is easy two to construct an efficient motion planner on~$S^2_+$ with two domains of continuity. Indeed, $\sigma$ is continuous on $S^2_+ \times S^2_+ \setminus A$, and in order to navigate on $A$ it suffices to fix orientation of the boundary circle. 
\end{example*}

We would also like to draw the reader's attention to the fact that the motion planner $\sigma_0 \colon\mathcal{V} \rightarrow X^I$ defined almost everywhere on $X \times X$ in the proof of Theorem~\ref{thm:main} has the following desirable properties: 
\begin{itemize}
\item If the initial and terminal states coincide, the output path is constant (cf. \cite{Iwase10,Iwase12}).
\item The path from $p$ to $q$ is the same as that from $q$ to $p$, only traversed in the opposite direction (cf. \cite{Farber07}).
\item Re-evaluating a motion in its middle does not change the choice of navigation arc, i.e. if $t_0\in I$ is the re-evaluation instant, then $$\sigma_0\big(\sigma_0(p,q)(t_0),q\big)(t)=\sigma_0(p,q)\big(t_0+(1-t_0)t\big).$$
\end{itemize}

\noindent The last property draws attention to the problem of algorithmically finding a vector in $T_pX$ pointing in the direction of a minimizing geodesic from $p$ to~$q$, rather than deciding on the whole motion at once. This approach highlights the concept of \emph{autonomy} of a mechanical system, allowing it to plan its motion \emph{on-the-fly}, perhaps making it possible to correct the path in case obstacles appear.\\

\noindent\textbf{Acknowledgements.} 
The authors thank the organizers of \textit{Workshop on TC and Related Topics}, held in Oberwolfach in March 2016, where this project was started. The first author gratefully acknowledges hospitality of Universit\'e catholique de Louvain during his stay in May 2016, where the bulk of this work was carried out. The second author acknowledges the Belgian Interuniversity Attraction Pole (IAP) for support within the framework ``Dynamics, Geometry and Statistical Physics'' (DYGEST).

\noindent \textsc{Zbigniew B\l{}aszczyk}\\
Faculty of Mathematics and Computer Science\\
Adam Mickiewicz University\\
Umultowska 87\\
60-479 Pozna\'n, Poland\\
\texttt{blaszczyk@amu.edu.pl}\medskip

\noindent \textsc{Jos\'e Gabriel Carrasquel-Vera}\\
Institut de Recherche en Math\'ematique et Physique\\
Universit\'e catholique de Louvain\\
2 Chemin du Cyclotron\\
1348 Louvain-la-Neuve, Belgium\\
\texttt{jose.carrasquel@uclouvain.be}
\end{document}